\documentclass[a4,12pt]{amsart}
\usepackage{amssymb,amstext,amsmath,amscd,amsthm,amsfonts,enumerate,graphicx,latexsym}
\usepackage[usenames]{color}
\newtheorem{thm}{Theorem}[section]
\newtheorem{cor}[thm]{Corollary}
\newtheorem{lem}[thm]{Lemma}
\newtheorem{prop}[thm]{Proposition}
\newtheorem*{thm*}{Theorem}
\newtheorem*{cor*}{Corollary}
\theoremstyle{definition}

\newtheorem*{claim*}{Claim}

\numberwithin{equation}{thm}
\def\dsg{\operatorname{\mathsf{D}_{sg}}}
\def\lcm{\operatorname{\mathsf{\underline{CM}}}}
\def\ann{\operatorname{\mathsf{Ann}}}
\def\db{\operatorname{\mathsf{D}^b}}
\def\mod{\operatorname{\mathsf{mod}}}
\def\ker{\operatorname{\mathsf{Ker}}}
\def\coker{\operatorname{\mathsf{Coker}}}
\def\D{\operatorname{\mathsf{D}}}
\def\rad{\operatorname{\mathsf{rad}}}

\def\height{\operatorname{ht}}
\def\LL{\operatorname{\ell\ell}}
\def\e{\operatorname{e}}
\def\depth{\operatorname{depth}}

\def\syz{\mathsf{\Omega}}
\def\H{\mathsf{H}}
\def\B{\mathsf{B}}
\def\Z{\mathsf{Z}}
\def\K{\mathsf{K}}

\def\Ext{\mathsf{Ext}}
\def\Hom{\mathsf{Hom}}

\def\N{\mathfrak{N}}
\def\m{\mathfrak{m}}
\def\p{\mathfrak{p}}

\def\A{\mathcal{A}}

\def\ZZ{\mathbb{Z}}

\def\xx{\text{{\boldmath$x$}}}
\begin{document}
\setlength{\baselineskip}{15pt}
\title{Upper bounds for dimensions of singularity categories}
%\date{at\ \digitalhour:\digitalminute\ (GMT-05:00)\ on\ \today}
\date{\today}
\author{Hailong Dao}
\address{Department of Mathematics, University of Kansas, Lawrence, KS 66045-7523, USA}
\email{hdao@math.ku.edu}
\urladdr{http://www.math.ku.edu/~hdao/}
\author{Ryo Takahashi}
\address{Graduate School of Mathematics, Nagoya University, Furocho, Chikusaku, Nagoya 464-8602, Japan/Department of Mathematics, University of Nebraska, Lincoln, NE 68588-0130, USA}
\email{takahashi@math.nagoya-u.ac.jp}
\urladdr{http://www.math.nagoya-u.ac.jp/~takahashi/}
\thanks{2010 {\em Mathematics Subject Classification.} Primary 13D09; Secondary 13C14, 18E30}
\thanks{{\em Key words and phrases.} dimension of triangulated category, singularity category, stable category of (maximal) Cohen-Macaulay modules}
\thanks{The first author was partially supported by NSF grants DMS 0834050 and DMS 1104017. The second author was partially supported by JSPS Grant-in-Aid for Young Scientists (B) 22740008 and by JSPS Postdoctoral Fellowships for Research Abroad}
%\subjclass[2000]{18E30, 16D90, 13C14}
%\keywords{resolving subcategory, thick subcategory, complete intersection dimension, Cohen-Macaulay module, dimension of triangulated category, Cohen-Macaulay representation type, Artin-Rees property}
%\dedicatory{Dedicated to Professor Craig Huneke on the occasion of his sixtieth birthday}
\begin{abstract}
This paper gives upper bounds for the dimension of the singularity category of a Cohen-Macaulay local ring with an isolated singularity.
One of them recovers an upper bound given by Ballard, Favero and Katzarkov in the case of a hypersurface.
\end{abstract}
\maketitle
%\tableofcontents
%%%%%%%%%%%%%%%%%%%%%%%%%%%%%%%%%%%%%%%%%%%%%%%%%%%%%%%%
\section{Results}

The notion of the dimension of a triangulated category has been introduced by Bondal, Rouquier and Van den Bergh \cite{BV,R}, which measures the number of extensions necessary to build the category out of a single object.
The singularity category $\dsg(R)$ of a noetherian ring/scheme $R$ is one of the most crucial triangulated categories.
This has been introduced by Buchweitz \cite{B} by the name of stable derived category.
There are many studies on singularity categories by Orlov \cite{O1,O2,O3,O4} in connection with the Homological Mirror Symmetry Conjecture.

It is a natural and fundamental problem to find upper bounds for the dimension of the singularity category of a noetherian ring.
In general, the dimension of the singularity category is known to be finite for large classes of excellent rings containing fields \cite{ddc,R}, but only a few explicit upper bounds have been found so far.
The Loewy length is an upper bound for an artinian ring \cite{R}, and so is the global dimension for a ring of finite global dimension \cite{C,KK}.
Recently, an upper bound for an isolated hypersurface singularity has been given \cite{BFK}.

The main purpose of this paper is to give upper bounds for a Cohen-Macaulay local ring with an isolated singularity.
The main result of this paper is the following theorem.

\begin{thm}\label{1.2}
Let $(R,\m,k)$ be a complete equicharacteristic Cohen-Macaulay local ring with $k$ perfect.
Suppose that $R$ is an isolated singularity.
Then the sum $\N^R$ of the Noether differents of $R$ is $\m$-primary.
Let $I$ be an $\m$-primary ideal of $R$ contained in $\N^R$.
\begin{enumerate}[\rm(1)]
\item
One has $\dsg(R)=\langle k\rangle_{(\nu(I)-\dim R+1)\LL(R/I)}$.
Hence there is an inequality
$$
\dim\dsg(R)\le(\nu(I)-\dim R+1)\LL(R/I)-1.
$$
\item
Assume that $k$ is infinite.
Then $\dsg(R)=\langle k\rangle_{\e(I)}$, and hence one has
$$
\dim\dsg(R)\le\e(I)-1.
$$
\end{enumerate}
\end{thm}

Here we explain the notation used in the above theorem.
Let $(R,\m,k)$ be a commutative noetherian complete equicharacteristic local ring.
Let $A$ be a Noether normalization of $R$, that is, a formal power series subring $k[[x_1,\dots,x_d]]$, where $x_1,\dots,x_d$ is a system of parameters of $R$.
Let $R^e=R\otimes_AR$ be the enveloping algebra of $R$ over $A$.
Define a map $\mu:R^e\to R$ by $\mu(a\otimes b)=ab$ for $a,b\in R$.
Then the ideal $\N^R_A=\mu(\ann_{R^e}\ker\mu)$ of $R$ is called the Noether different of $R$ over $A$.
We denote by $\N^R$ the sum of $\N^R_A$, where $A$ runs through the Noether normalizations of $R$.
For an $\m$-primary ideal $I$ of $R$, let $\nu(I)=\dim_k(I\otimes_Rk)$ be the minimal number of generators of $I$ and $\e(I)=\lim_{n\to\infty}\frac{d!}{n^d}\ell(R/I^{n+1})$ the multiplicity of $I$.
The Loewy length of an artinian ring $\Lambda$ is denoted by $\LL(\Lambda)$, that is, the minimum positive integer $n$ with $(\rad\Lambda)^n=0$.

Our Theorem \ref{1.2} yields the following result.

\begin{cor}\label{1.3'}
Let $k$ be a perfect field, and let $R=k[[x_1,\dots,x_n]]/(f_1,\dots,f_m)$ be a Cohen-Macaulay ring having an isolated singularity.
Let $J$ be the Jacobian ideal of $R$, namely, the ideal generated by the $h$-minors of the Jacobian matrix $(\frac{\partial f_i}{\partial x_j})$, where $h=\height(f_1,\dots,f_m)$.
\begin{enumerate}[\rm(1)]
\item
One has $\dsg(R)=\langle k\rangle_{(\nu(J)-\dim R+1)\LL(R/J)}$.
Hence there is an inequality $\dim\dsg(R)\le(\nu(J)-\dim R+1)\LL(R/J)-1$.
\item
If $k$ is infinite, then $\dsg(R)=\langle k\rangle_{\e(J)}$, and it holds that $
\dim\dsg(R)\le\e(J)-1$.
\end{enumerate}
\end{cor}

Corollary \ref{1.3'} immediately recovers the following result, which is stated in \cite{BFK}.

\begin{cor}[Ballard-Favero-Katzarkov]
Let $k$ be an algebraically closed field of characteristic zero.
Let $R=k[[x_1,\dots,x_n]]/(f)$ be an isolated hypersurface singularity.
Then $\dsg(R)=\langle k\rangle_{2\LL(R/(\frac{\partial f}{\partial x_1},\dots,\frac{\partial f}{\partial x_n})R)}$, and hence $\dim\dsg(R)\le2\LL(R/(\textstyle\frac{\partial f}{\partial x_1},\dots,\frac{\partial f}{\partial x_n})R)-1$.
\end{cor}

As another application of Theorem \ref{1.2}, we obtain upper bounds for the dimension of the stable category $\lcm(R)$ of maximal Cohen-Macaulay modules over an excellent Gorenstein ring $R$:

\begin{cor}\label{1.3}
Let $R$ be an excellent Gorenstein equicharacteristic local ring with perfect residue field $k$, and assume that $R$ is an isolated singularity.
Then $\N^{\widehat R}$ is an $\widehat{\m}$-primary ideal of the completion $\widehat R$ of $R$.
Let $I$ be an $\widehat{\m}$-primary ideal contained in $\N^{\widehat R}$.
Put $d=\dim R$, $n=\nu(I)$, $l=\LL({\widehat R}/I)$ and $e=\e(I)$.
\begin{enumerate}[\rm(1)]
\item
One has $\lcm(R)=\langle\syz^dk\rangle_{(n-d+1)l}$, and $\dim\lcm(R)\le(n-d+1)l-1$.
\item
If $k$ is infinite, then $\lcm(R)=\langle\syz^dk\rangle_e$, and one has $\dim\lcm(R)\le e-1$.
\end{enumerate}
\end{cor}

%%%%%%%%%%%%%%%%%%%%%%%%%%%%%%%%%%%%%%%%%%%%%%%%%%%%%%%%
\section{Proofs}

This section is devoted to proving our results stated in the previous section.
For the definition of the dimension of a triangulated category and related notation, we refer the reader to \cite[Definition 3.2]{R}.
We denote by $\D(\A)$ the derived category of an abelian category $\A$.
Let $\H^iX$ (respectively, $\Z^iX$, $\B^iX$) denote the $i$-th homology (respectively, cycle, boundary) of a complex $X$ of objects of $\A$, and set $\H X=\bigoplus_{i\in\ZZ}\H^iX$.

\begin{lem}\label{2.1}
Let $\A$ be an abelian category and $X$ a complex of objects of $\A$.
\begin{enumerate}[\rm(1)]
\item
Let $n$ be an integer.
If $\H^iX=0$ for all $i>n$, then there exists an exact triangle
$$
Y \to X \to \H^nX[-n] \rightsquigarrow
$$
in $\D(\A)$ such that $\H^iY\cong
\begin{cases}
0 & (i\ge n)\\
\H^iX & (i<n)
\end{cases}$.
\item
Let $n\ge m$ be integers.
If $\H^iX=0$ for all $i>n$ and $i<m$, then $X\in\langle\H X\rangle_{n-m+1}^{\D(\A)}$.
\end{enumerate}
\end{lem}

\begin{proof}
(1) Truncating $X=(\cdots\xrightarrow{d^{i-1}} X^i \xrightarrow{d^i} X^{i+1} \xrightarrow{d^{i+1}} \cdots)$, we get complexes
\begin{align*}
X' & =(\cdots \xrightarrow{d^{n-2}} X^{n-1} \xrightarrow{d^{n-1}} \Z^nX \to 0),\\
Y & =(\cdots \xrightarrow{d^{n-2}} X^{n-1} \xrightarrow{d^{n-1}} \B^nX \to 0).
\end{align*}
There are natural morphisms $Y\xrightarrow{f}X'\xrightarrow{g}X$, where $f$ is a monomorphism and $g$ is a quasi-isomorphism.
We have a short exact sequence $0 \to Y \xrightarrow{f} X' \to \H^nX[-n] \to 0$ of complexes, which induces an exact triangle as in the assertion.

(2) Applying (1) repeatedly, for each $0\le j\le n-m$ we obtain an exact triangle
$$
X_{j+1} \to X_j \to \H^{n-j}X[-(n-j)] \rightsquigarrow
$$
in $\D(\A)$ with $X_0=X$ such that $\H^iX_j\cong0$ for $i>n-j$ and $\H^iX\cong\H^iX$ for $i\le n-j$.
Hence $X_{n-m+1}\cong0$ in $\D(\A)$, which implies that $X_{n-m}$ is in $\langle\H^mX\rangle$.
Inductively, we observe that $X=X_0$ belongs to $\langle\H^mX\oplus\H^{m+1}X\oplus\cdots\oplus\H^nX\rangle_{n-m+1}=\langle\H X\rangle_{n-m+1}$.
\end{proof}

For a commutative noetherian ring $R$, we denote by $\mod R$ the category of finitely generated $R$-modules, and by $\db(\mod R)$ the bounded derived category of $\mod R$.
For a sequence $\xx=x_1,\dots,x_n$ of elements of $R$ and an $R$-module $M$, let $\K(\xx,M)$ denote the Koszul complex of $\xx$ on $M$.

\begin{prop}\label{2.2}
Let $(R,\m)$ be a commutative noetherian local ring and $I$ an $\m$-primary ideal of $R$.
Let $\xx=x_1,\dots,x_n$ be a sequence of elements of $R$ that generates $I$.
Then for any finitely generated $R$-module $M$ one has $\K(\xx,M)\in\langle k\rangle_{(n-t+1)l}$ in $\db(\mod R)$, where $t=\depth M$ and $l=\LL(R/I)$.
\end{prop}

\begin{proof}
Set $\K(\xx,M)=K=(0\to K^{-n}\to\cdots\to K^0\to0)$.
By \cite[Proposition 1.6.5(b)]{BH}, each homology $H^i=\H^iK$ is annihilated by $I$, and $H^i$ is regarded as a module over $R/I$.
There is a filtration $0=\m^l(R/I)\subsetneq\cdots\subsetneq\m(R/I)\subsetneq R/I$ of ideals of $R/I$.
For each integer $i$ we have a filtration
$$
0=\m^lH^i\subseteq\cdots\subseteq\m H^i\subseteq H^i
$$
of submodules of $H^i$, which shows $H^i\in\langle k\rangle_l$ in $\db(\mod R)$.
We see from \cite[Theorem 1.6.17(b)]{BH} that $H^i=0$ for all $i<t-n$ and $i>0$.
It follows from Lemma \ref{2.1}(2) that $K$ is in $\langle\textstyle\bigoplus_{i=t-n}^0H^i\rangle_{n-t+1}$ in $\db(\mod R)$, which is contained in $\langle k\rangle_{(n-t+1)l}$.
\end{proof}

Recall that the singularity category $\dsg(R)$ of a (commutative) noetherian ring $R$ is defined as the Verdier quotient of $\db(\mod R)$ by the full subcategory of perfect complexes.
(A perfect complex is by definition a bounded complex of finitely generated modules.)

\begin{prop}\label{2.3}
Let $R$ be a commutative noetherian ring, and let $M$ be a finitely generated $R$-module.
Let $\xx=x_1,\dots,x_n$ be a sequence of elements of $R$ such that the multiplication map $M\xrightarrow{x_i}M$ is a zero morphism in $\dsg(R)$ for every $1\le i\le n$.
Then $M$ is isomorphic to a direct summand of $\K(\xx,M)$ in $\dsg(R)$.
\end{prop}

\begin{proof}
By definition the Koszul complex $\K(x_i,M)=(0 \to M \xrightarrow{x_i} M \to 0)$ is the mapping cone of the multiplication map $M\xrightarrow{x_i}M$, and there is an exact triangle $M \xrightarrow{x_i} M \to \K(x_i,M) \rightsquigarrow$ in $\dsg(R)$.
By assumption, we have an isomorphism $M\oplus M[1]\cong\K(x_i,M)=\K(x_i,R)\otimes_RM$ in $\dsg(R)$.
We observe that
\begin{align*}
\K(\xx,M)&=\K(x_1,R)\otimes_R\cdots\otimes_R\K(x_{n-1},R)\otimes_R(\K(x_n,R)\otimes_RM)\\
&\gtrdot\K(x_1,R)\otimes_R\cdots\otimes_R\K(x_{n-2},R)\otimes_R(\K(x_{n-1},R)\otimes_RM)\\
&\cdots\\
&\gtrdot\K(x_1,R)\otimes_RM\gtrdot M,
\end{align*}
where $A\gtrdot B$ means that $A$ has a direct summand isomorphic to $B$ in $\dsg(R)$.
\end{proof}

\begin{lem}\label{2.4}
\begin{enumerate}[\rm(1)]
\item
Let $\A$ be an abelian category.
Let $P=(\cdots \xrightarrow{d^{b-1}} P^b \xrightarrow{d^b} \cdots \xrightarrow{d^{a-1}} P^a\to0)$ be a complex of projective objects of $\A$ with $\H^iP=0$ for all $i<b$.
Then one has an exact triangle
$$
F \to P \to C[-b] \rightsquigarrow
$$
in $\D(\A)$, where $F=(0 \to P^{b+1} \xrightarrow{d^{b+1}} \cdots \xrightarrow{d^{a-1}} P^a \to 0)$ and $C=\coker d^{b-1}$.
\item
Let $R$ be a commutative noetherian ring.
\begin{enumerate}[\rm(a)]
\item
For any $X\in\dsg(R)$ there exist $M\in\mod R$ and $n\in\ZZ$ such that $X\cong M[n]$ in $\dsg(R)$.
\item
Let $M$ be a finitely generated $R$-module.
Then for an integer $n\ge0$ there exists an exact triangle
$$
F \to M \to \syz^nM[n] \rightsquigarrow
$$
in $\db(\mod R)$, where $F=(0\to F^{-(n-1)} \to \cdots \to F^0 \to 0)$ is a perfect complex.
\end{enumerate}
\end{enumerate}
\end{lem}

\begin{proof}
(1) There is a short exact sequence $0\to F\to P\to Q\to 0$ of complexes, where $Q=(\cdots \xrightarrow{d^{b-2}} P^{b-1} \xrightarrow{d^{b-1}} P^b \to 0)$.
Then $Q\cong C[-b]$ in $\D(\A)$.

(2) The assertion (a) is immediate from (1).
Setting $a=0\ge-n=b$ and letting $P$ be a projective resolution of $M$ in (1) implies (b).
\end{proof}

Recall that a commutative noetherian ring $R$ is called an isolated singularity if the local ring $R_\p$ is regular for every nonmaximal prime ideal $\p$ of $R$.

\begin{prop}\label{2.5}
Let $R$ be a complete equicharacteristic Cohen-Macaulay local ring.
Then for an element $x\in\N^R$ and a maximal Cohen-Macaulay $R$-module $M$, the multiplication map $M\xrightarrow{x}M$ is a zero morphism in $\dsg(R)$.
\end{prop}

\begin{proof}
Lemma \ref{2.4}(2) implies that there is an exact triangle
$$
F \xrightarrow{f} M \xrightarrow{g} \syz M[1] \rightsquigarrow
$$
in $\db(\mod R)$, where $F$ is a finitely generated free $R$-module.
By virtue of \cite[Corollary 5.13]{W}, the ideal $\N^R$ annihilates $\Ext_R^1(M,\syz M)=\Hom_{\db(\mod R)}(M,\syz M[1])$.
Hence $xg=0$ in $\db(\mod R)$, and there exists a morphism $h:M\to F$ such that $fh=(M\xrightarrow{x}M)$ in $\db(\mod R)$.
Send this equality by the localization functor $\db(\mod R)\to\dsg(R)$, and note that $F\cong0$ in $\dsg(R)$.
Thus the multiplication map $M\xrightarrow{x}M$ is zero in $\dsg(R)$.
\end{proof}

Now we can prove the results given in the previous section.

\begin{proof}[{\bf Proof of Theorem \ref{1.2}}]
As $k$ is a perfect field and $R$ is an isolated singularity, $\N^R$ is $\m$-primary by \cite[Lemma (6.12)]{Y}.

(1) Put $d=\dim R$, $n=\nu(I)$, $l=\LL(R/I)$ and $e=\e(I)$.
We have $I=(\xx)$ for some sequence $\xx=x_1,\dots,x_n$ of elements in $I$.
Let $X\in\dsg(R)$.
Then, using Lemma \ref{2.4}(2), we see that $X\cong\syz^dN[n]$ for some $N\in\mod R$ and $n\in\ZZ$.
Note that $M:=\syz^dN$ is a maximal Cohen-Macaulay $R$-module.
Proposition \ref{2.2} implies that $\K(\xx,M)$ belongs to $\langle k\rangle_{(n-d+1)l}$ in $\db(\mod R)$.
Applying the localization functor $\db(\mod R)\to\dsg(R)$, we have $\K(\xx,M)\in\langle k\rangle_{(n-d+1)l}$ in $\dsg(R)$.
Since $M$ is isomorphic to a direct summand of $\K(\xx,M)$ in $\dsg(R)$ by Propositions \ref{2.3} and \ref{2.5}, we get $M\in\langle k\rangle_{(n-d+1)l}$ in $\dsg(R)$.
Therefore $\dsg(R)=\langle k\rangle_{(n-d+1)l}$ follows.

(2) Since $k$ is infinite, there exists a parameter ideal $Q$ of $R$ that is a reduction of $I$ (cf. \cite[Corollary 4.6.10]{BH}).
We have
$$
(\nu(Q)-\dim R+1)\LL(R/Q)=\LL(R/Q)\le\ell(R/Q)=\e(Q)=\e(I).
$$
The assertion is a consequence of (1).
\end{proof}

\begin{proof}[{\bf Proof of Corollary \ref{1.3'}}]
We see from \cite[Lemmas 4.3, 5.8 and Propositions 4.4, 4.5]{W} that $J$ is contained in $\N^R$ and defines the singular locus of $R$.
Hence the assertion follows from Theorem \ref{1.2}.
\end{proof}

\begin{proof}[{\bf Proof of Corollary \ref{1.3}}]
We notice that $\widehat R$ is an isolated singularity.
Suppose that $\dsg(\widehat R)=\langle k\rangle_r$ holds for some $r\ge0$.
Then it follows from \cite[Theorem 4.4.1]{B} that $\lcm(\widehat R)=\langle\syz_{\widehat R}^dk\rangle_r=\langle\widehat{\syz_R^dk}\rangle_r$.
The proof of \cite[Theorem 5.8]{ddc} shows that $\lcm(R)=\langle\syz_R^dk\rangle_r$.
Thus, Theorem \ref{1.2} completes the proof.
\end{proof}

%%%%%%%%%%%%%%%%%%%%%%%%%%%%%%%%%%%%%%%%%%%%%%%%%%%%%%%%
\section*{Acknowledgments}
The authors thank Luchezar Avramov and Srikanth Iyengar for their valuable comments.
%%%%%%%%%%%%%%%%%%%%%%%%%%%%%%%%%%%%%%%%%%%%%%%%%%%%%%%%

%%%%%%%%%%%%%%%%%%%%%%%%%%%%%%%%%%%%%%%%%%%%%%%%%%%%%%%%
\end{document}